\documentclass[12pt]{article}
\usepackage{amsmath,amsfonts,amssymb,amsthm}
\usepackage{color}
\usepackage{mathtools}

\usepackage{geometry}
\usepackage{fancyhdr}
\usepackage{mathrsfs}
\usepackage{enumitem}

\usepackage{array}
\usepackage{graphicx}
\usepackage{float}
\usepackage{appendix}
\usepackage[all]{xy}
\usepackage{hyperref}
\usepackage{verbatim}

\input amssym.def
\begin{document}
		\def \Z{ \mathbb{Z}}
		\def \C{\mathbb {C}}
		\def \R{\mathbb {R}}
		\def \Q{\mathbb {Q}}
	   \def \N{\mathbb {N}}
	   \def \g{\mathfrak{g}}
	   \def \h{\mathfrak{h}}
	   \def \c {\mathbf{c}}
       \def \Y{\mathcal{Y}}
        \def \L{\mathcal{L}}

	\def \be{\begin{equation}\label}
		\def \ee{\end{equation}}
	\def \bex{\begin{example}\label}
		\def \eex{\end{example}}
	\def \bl{\begin{lem}\label}
		\def \el{\end{lem}}
	\def \bt{\begin{thm}\label}
		\def \et{\end{thm}}
	\def \bp{\begin{prop}\label}
		\def \ep{\end{prop}}
	\def \br{\begin{rem}\label}
		\def \er{\end{rem}}
	\def \bc{\begin{coro}\label}
		\def \ec{\end{coro}}
	\def \bd{\begin{de}\label}
		\def \ed{\end{de}}
	\newcommand{\m}{\bf m}
	\newcommand{\n}{\bf n}
	\newcommand{\nno}{\nonumber}
	\newcommand{\nord}{\mbox{\scriptsize ${\circ\atop\circ}$}}
	\newtheorem{thm}{Theorem}[section]
	\newtheorem{prop}[thm]{Proposition}
	\newtheorem{coro}[thm]{Corollary}
	\newtheorem{conj}[thm]{Conjecture}
	\newtheorem{example}[thm]{Example}
	\newtheorem{lem}[thm]{Lemma}
	\newtheorem{rem}[thm]{Remark}
	\newtheorem{de}[thm]{Definition}
	\newtheorem{hy}[thm]{Hypothesis}
	\makeatletter \@addtoreset{equation}{section}
	\def\theequation{\thesection.\arabic{equation}}
	\makeatother \makeatletter
	\begin{center}
		{\Large \bf  On integral forms for vertex  superalgebras associated
with affine Lie superalgebras and their modules}
	\end{center}

	\begin{center}		
{Hongyan Guo$^{a,b}$
			\footnote{Partially supported by
				NSFC (Nos. 11901224, 12371027),
and supported by the Fundamental Research Funds for the Central Universities (grant no. XJ2026004301).},
		Hongju Zhao$^{c}$	
		\\
		$\mbox{}^{a}$ School of Mathematics and Statistics,
			and Hubei Key Lab-Math. Sci.,
			Central China Normal University, Wuhan 430079, China\\
			$\mbox{}^{b}$ Key Lab NAA-MOE,
			Central China Normal University, Wuhan 430079, China\\
       $\mbox{}^{c}$ School of Mathematical Sciences, Harbin Normal University, Harbin 150025, China}
\end{center}

\begin{abstract}
This paper studies integral forms for affine vertex superalgebras and their modules.
We first obtain integral forms for the universal enveloping superalgebra
$U(\hat{\g})$ of an affine Lie superalgebra $\hat{\g}$.
For $\g$ a basic classical Lie superalgebra,
we give another construction of integral forms for $U(\hat{\g})$
using Chevalley basis, generalizing Kostant-Garland integral form theory.
Then we apply the theory to construct integral forms for
vertex (operator) superalgebras based on
 affine Lie superalgebras and their modules,
we also
investigate
when an integral form
contains the conformal vector,
and integral forms in contragredient modules for vertex operator superalgebras.
	\end{abstract}

\section{Introduction}

Vertex (operator) superalgebras and their modules are natural
 generalizations of vertex (operator) algebras and their modules,
and play an important role in (higher dimensional) quantum field theory,
especially in the study of the 4d/2d duality theory which has aroused considerable attention recently.
Previously, the study of vertex superalgebras mainly focused on the field
of characteristic 0.
But the notions of vertex superalgebras and their modules
work over any commutative ring,
and so it is natural to consider vertex superalgebras over $\Z$.
In particular, it is natural to look for $\Z$-forms of vertex superalgebras,
by analogy with the construction of Lie superalgebras over $\Z$ using Chevalley bases.

Integral forms of vertex algebras and their modules have been
studied widely (cf. \cite{B1},
\cite{C}, \cite{DG}, \cite{GZ}, \cite{M1}, \cite{M2}, etc.).
In this paper,
we generalize related notions and results to integral forms of vertex (operator) superalgebras
and their modules,
and integral forms for vertex superalgebras constructed from affine Lie superalgebras
are studied in detail.
In the process, integral forms for the universal enveloping superalgebra
of an affine Lie superalgebra are constructed and will be used to construct
integral forms for affine vertex superalgebras.

In Section \ref{sec:2},
we study integral forms for the universal enveloping superalgebras
of affine Lie superalgebras.
  Let $\g$ be a finite-dimensional simple Lie superalgebra
  with a non-degenerate even supersymmetric invariant bilinear form.
  Let $\hat{\g}=\g\otimes\C[t,t^{-1}]\oplus \C\c$ be the associated affine Lie superalgebra.
  If $\g$ has an integral form $\g_{\Z}$,
  then a general construction of an integral form $U_{\Z}(\hat{\g})$ for the universal enveloping superalgebra $U(\hat{\g})$
  is given in Proposition \ref{prop:2.5}.
  For $\g$ a basic classical Lie superalgebra,
      except for types $A(1,1)$ and $D(2,1;a)$ with $a\notin \Z$,
 we provide an integral form of the universal enveloping superalgebra $U(\hat{\g})$
  with the help of a Chevalley basis of $\g$.
      This is a generalization of the Kostant-Garland $\Z$-form (including divided powers)
      for superalgebras (cf. \cite{FG}, \cite{G}, \cite{M}, \cite{P}, etc.).
More precisely, let $\hat{\g}$ be an affine Lie superalgebra with $\g$ the aforementioned basic classical Lie superalgebra.
Then $\hat{\g}$ has a Chevalley basis
	 $$\mathcal{C}:=\{X_{\alpha}(m)\mid \alpha\in \Delta\}\sqcup \{H_{i}(n)\mid (i,n)\in \hat{I}\},$$
 where $X_{\alpha}$, $\alpha\in\Delta$ and $H_{i}$, $i=1,\ldots,\ell$ form a Chevalley basis of $\g$.
 Let $K_{\Z}(\hat{\g})$ be the $\Z$-subalgebra of $U(\hat{\g})$ generated by the elements
\begin{eqnarray*}
{X_{\beta}(m)^{(r)}},\; \Lambda_{s}(H_{i}(jm')_{j=1}^{s}), \, X_{\gamma}(m'')
\end{eqnarray*}
for $\beta \in \Delta_{\bar{0}}$, $\gamma\in \Delta_{\bar{1}}$,
$(i,m')\in \hat{I}$,
  $m,m''\in \Z$, and $r,s\in \Z_+$.
 We show that $K_{\Z}(\hat{\g})$ is an integral form of the universal enveloping superalgebra $U(\hat{\g})$.
On the other hand,
as we need, we consider
 the $\Z$-subalgebra of $U(\hat{\g})$ generated by elements
	\begin{eqnarray*}
X_{\beta}(m)^{(r)},\;\; X_{\gamma}(m')
\end{eqnarray*}
	for $\beta \in \Delta_{\bar{0}}$, $\gamma\in \Delta_{\bar{1}}$,
   $m,m'\in \Z$, and $r\in \Z_+$.
Denote the subalgebra by $\tilde{U}_{\Z}(\hat{\g})$.
Then $\tilde{U}_{\Z}(\hat{\g})\subset K_{\Z}(\hat{\g})$ is also an integral form of $U(\hat{\g})$.

  In Section \ref{sec:3}, we first state some basic facts about
 integral forms of vertex superalgebras.
  Then we construct integral forms for vertex superalgebras
  based on affine Lie superalgebras,
  using the results of integral forms of the universal
  enveloping superalgebra of an affine Lie superalgebra constructed in Section \ref{sec:2}.
  For these integral forms, we exhibit a natural set of generating elements
  for affine vertex superalgebras and modules of arbitrary integral level
   (Proposition \ref{(prop3.4)} and Theorem \ref{(th3.9)}).
  Besides,
  we give conditions under which an integral form of a vertex operator superalgebra
  can be extended to include a multiple of the conformal vector $\omega$ (Theorem \ref{thomeg}).
  At the end of this section, we consider
  the construction of integral forms in contragredient modules,
  applying this to the vertex superalgebra which has a non-degenerate invariant bilinear form
  (Propositions \ref{prop4.3} and \ref{prop4.5}).

Throughout this paper, we use $\C$, $\Z$, $\Z_{+}$ and $\N$ to
denote the sets of complex numbers, integers, non-negative integers, and positive integers, respectively.

\section{Integral forms for the universal enveloping superalgebra of affine Lie superalgebras}\label{sec:2}

In this section, we construct integral forms for the universal enveloping superalgebra
 $U(\hat{\g})$ of the affine Lie superalgebra $\hat{\g}$,
 where $\g$ is a finite-dimensional simple Lie superalgebra
 with a non-degenerate even supersymmetric invariant bilinear form  $(\cdot,\cdot)$.
Besides, for $\g$ a basic classical Lie superalgebra
(excluding types $A(1,1)$ and $D(2,1;a)$ with $a\notin \Z$),
we give another construction of integral forms for $U(\hat{\g})$
through a Chevalley basis of $\g$ and divided powers of elements of $\hat{\g}$.

\subsection{The case of general Lie superalgebras}

Let $V=V_{\bar{0}}\oplus V_{\bar{1}}$ be a super vector space.
Given nonzero $u\in V_{\bar{i}}$, let the parity be $|u|=i$, $i\in\{0,1\}$.
We say $V_{\Z}\subseteq V$ is an {\em integral form} of $V$ as a super vector space
  if
  it is an integral form of $V$ as a vector space,
  i.e. $V_{\Z}\otimes_{\Z}\C\cong V$ as a vector space,
  $V_{\Z}$ is $\Z_2$-graded
 $V_{\Z}=(V_{\Z})_{\bar{0}}\oplus (V_{\Z})_{\bar{1}}$,
 and $V_{\Z}$ is compatible with the $\Z_2$-grading of $V$
  in the sense that
  \begin{eqnarray*}
  (V_{\Z})_{\lambda}  = (V_{\lambda})_{\Z} (:= V_{\Z} \cap V_{\lambda})\;\;\text{for}\; \lambda \in \Z_2.
  \end{eqnarray*}

\begin{de}
{\em
Let $A$ be a (Lie or associative) superalgebra.
An {\em integral form $A_{\Z}$} of $A$ is an integral form of $A$ as a super vector space
$A_{\Z}=(A_{\Z})_{\bar{0}}\oplus(A_{\Z})_{\bar{1}}$,
and $(A_{\Z})_{\lambda}(A_{\Z})_{\mu}\subseteq (A_{\Z})_{\lambda+\mu}$
for $\lambda,\mu\in\Z_{2}$.

Let $M$ be an $A$-module. Then $M_{\Z}$ is called an {\em integral form of $M$}
if $M_{\Z}$ is an integral form of $M$ as a super vector space
$M_{\Z}=(M_{\Z})_{\bar{0}}\oplus(M_{\Z})_{\bar{1}}$,
and the module action satisfies
$(A_{\Z})_{\lambda}(M_{\Z})_{\mu}\subseteq (M_{\Z})_{\lambda+\mu}$
for $\lambda,\mu\in\Z_{2}$.
}
\end{de}

Let $(\g,[\cdot,\cdot])$ be a finite-dimensional simple Lie superalgebra with a non-degenerate
even supersymmetric invariant bilinear form $(\cdot,\cdot)$.
Then we have an affine Lie superalgebra
\begin{eqnarray}
\hat{\g}=\g\otimes\C[t,t^{-1}]\oplus\C \c
\end{eqnarray}
with bracket
\begin{eqnarray}\label{eq:brac}
[a\otimes t^{m}, b\otimes t^{n}]=[a,b]\otimes t^{m+n}+m\delta_{m+n,0}(a,b)\c
\end{eqnarray}
for $a,b\in\g$, $m,n\in\Z$, where $\c$ is a central element of $\hat{\g}$.
Consider the triangular decomposition
$\hat{\g}=\hat{\g}_{-}\oplus \hat{\g}_{0}\oplus \hat{\g}_{+}$,
where
\begin{equation}
	\hat{\g}_{\pm}=\g \otimes t^{\pm 1}\C[t^{\pm 1}], \,\,\;
 \hat{\g}_{0}=\g\oplus \C\c,
\end{equation}
and let $\hat{\g}_{(\ge 0)} =\hat{\g}_{+}\oplus \hat{\g}_{0}$.
There is a natural $\Z_{2}$-grading on $\hat{\g}$ by
\begin{eqnarray}
\hat{\g}_{\bar{0}}=\g_{\bar{0}}\otimes\C[t,t^{-1}]\oplus\C\c, \quad
\hat{\g}_{\bar{1}}=\g_{\bar{1}}\otimes\C[t,t^{-1}].
\end{eqnarray}

Suppose $(\g, (\cdot,\cdot))$ has an integral form $\g_{\Z}$,
i.e. $\g_{\Z}$ is an integral form of $\g$ as a Lie superalgebra
 and the bilinear form $(\cdot,\cdot)$ is integer-valued on $\g_{\Z}$.
Then $\hat{\g}$ naturally has an integral form
\begin{equation}\label{eq:2.7}
	\hat{\g}_{\Z}= \g_{\Z}\otimes_{\Z}\Z[t,t^{-1}]\oplus \Z\mathbf{c},
\end{equation}
with $\Z_2$-grading
\begin{equation}
	(\hat{\g}_{\Z})_{\bar{0}}=(\g_{\Z})_{\bar{0}}\otimes_{\Z}\Z[t,t^{-1}]\oplus\Z\c,\,\,
	(\hat{\g}_{\Z})_{\bar{1}}=(\g_{\Z})_{\bar{1}}\otimes_{\Z}\Z[t,t^{-1}].
\end{equation}

For $a\in \g$ and $m\in \Z$,
we denote $a(m)=a\otimes t^m\in \hat{\g}$.
The following result of an integral form
of the universal enveloping superalgebra $U(\hat{\g})$
of the affine Lie superalgebra $\hat{\g}$ is straightforward to check.

\begin{prop}\label{prop:2.5}
Suppose $\{a_j\mid j=1,\ldots,r\}$ and $\{b_{j^{\prime}}\mid j^{\prime}=1,\ldots,s\}$ are $\Z$-bases
of $({\g}_{\Z})_{\bar{0}}$ and $(\g_{\Z})_{\bar{1}}$, respectively.
Assume that the bilinear form $(\cdot,\cdot)$ is integer-valued on the basis.
 Let $U_\Z(\hat{\g})$ be the $\Z$-span of elements of the form
	\begin{equation}\label{eq:pbw Z-basis}
\c ^{k}a_{j_1}(m_1)^{k_1}\cdots a_{j_l}(m_l)^{k_l}b_{j^{\prime}_1}(n_1)\cdots b_{j^{\prime}_{p}}(n_p)
		\end{equation}
for distinct $a_{j_1}(m_1), \ldots, a_{j_l}(m_l)\in (\hat{\g}_{\Z})_{\bar{0}}$
and distinct $b_{j^{\prime}_1}(n_1),\ldots, b_{j^{\prime}_{p}}(n_p)\in (\hat{\g}_{\Z})_{\bar{1}}$,
where $k\in \Z_+$, $1\le j_1\le \ldots\le j_l\le r$, $k_q\in \N$,
$m_1\le \ldots \le m_r\in \Z$, $1\le j^{\prime}_1\le \ldots\le j^{\prime}_p\le s$
and $n_1\le \ldots \le n_p\in \Z$.
Then $U_\Z(\hat{\g})$
	is an integral form of $U(\hat{\g})$ as an associative superalgebra
with a $\Z$-basis consisting of elements of the form (\ref{eq:pbw Z-basis}),
and for $i\in\{0,1\}$, $U_{\Z}(\hat{\g})_{\bar{i}}$ is the $\Z$-span of
elements of the form (\ref{eq:pbw Z-basis}) with $p\equiv i\pmod{2}$.
\end{prop}

Consider the integral form $U_{\Z}(\hat{\g})$ constructed in Proposition \ref{prop:2.5}.
Let $U_{\Z}(\hat{\g}_{\pm})$ be the $\Z$-span of ordered
  products in $a_j(\pm m),b_{j^{\prime}}(\pm n)$
 with nondecreasing order of powers of $t^{\pm 1}$,
 where $j=1,\ldots,r$, $j^{\prime}=1,\ldots,s$, $m, n\in \N$.
 Let $U_{\Z}(\hat{\g}_{0})$ be the $\Z$-span of ordered products in
  $\c$, $a_j(0)$, $b_{j^{\prime}}(0)$ for $j=1,\ldots,r$, $j^{\prime}=1,\ldots,s$.
   Then $U_{\Z}(\hat{\g}_{\pm})$ and $U_{\Z}(\hat{\g}_{0})$ are
   integral forms of $U(\hat{\g}_{\pm})$ and $U(\hat{\g}_{0})$, respectively.
   We order the products (\ref{eq:pbw Z-basis}) in such a way that
   \begin{equation}\label{eq3.12}
	U_{\Z}(\hat{\g}) = U_{\Z}(\hat{\g}_{-}) U_{\Z}(\hat{\g}_{0}) U_{\Z}(\hat{\g}_{+}).
\end{equation}
Clearly, $U_{\Z}(\hat{\g}_{\pm})$ is $\Z$-graded
\begin{equation}\label{eq:2.15}
	U_{\Z}(\hat{\g}_{\pm})=\coprod_{n\in \Z}( U_{\Z}(\hat{\g}_{\pm}) \cap U(\hat{\g}_{\pm})_{n}),
\end{equation}
where  $U(\hat{\g}_{\pm})_{n}$ is the linear span of elements of the form
\begin{equation}
	a_{j_1}(m_{1})^{k_1}\cdots a_{j_l}(m_{l})^{k_l}b_{j^{\prime}_1}(n_{1})\cdots b_{j^{\prime}_p}(n_{p})
\end{equation}
for $k_q\in \N$ with $k_1 m_1 +\cdots +k_lm_l+n_1+\cdots+n_p=n$.

\subsection{The case of basic classical Lie superalgebras}\label{sec:2.2}

Recall the definition and related results of basic classical Lie superalgebras from \cite{Kac}, \cite{Kac2}, etc.
Let $\g=\g_{\bar{0}}\oplus \g_{\bar{1}}$ be a basic classical Lie superalgebra,
and fix a Cartan subalgebra $\h\subseteq \g_{\bar{0}}$.
Consider the root space decomposition
\begin{equation}
\g={\h}\oplus \bigoplus_{\alpha\in \Delta}\g_{\alpha},
\end{equation}
where $\g_{\alpha}=\{x\in \g \mid [h,x]=\alpha(h)x \text{ for all }h\in \h\}$
 is the root space associated to the root $\alpha$,
 and $\Delta=\Delta_{\bar{0}}\cup \Delta_{\bar{1}}$ is the root system with
\begin{equation}
	\begin{aligned}
		&\Delta_{\bar{0}}=\{\alpha\in \h^*\setminus\{0\}\mid\g_{\alpha}\cap \g_{\bar{0}}\ne 0\}\;\;(even\; roots),\\
		&\Delta_{\bar{1}}=\{\alpha\in \h^*\mid \g_{\alpha}\cap \g_{\bar{1}}\ne 0\}\;\;(odd\; roots).
	\end{aligned}
\end{equation}

For $\alpha \in \Delta$,
let $H_{\alpha}$ be the coroot associated with $\alpha$ (see \cite[Definition 2.20]{IK}).
Then
$H_{-\alpha}=-H_{\alpha}$, $H_{2\alpha}=\frac{1}{2}H_{\alpha}$ if $2\alpha \in \Delta$,
and (see \cite[Lemma 2.23]{IK})
\begin{equation}\label{eq:2.23}
 		\alpha(H_{\alpha})=\begin{cases}
 		0,\;\; \text{if}\;\; (\alpha,\alpha ) =0 ;\\
 		2,\;\;\text{if}\;\;( \alpha,\alpha ) \ne0.
 		\end{cases}
 	\end{equation}

  Fix a distinguished simple root system
 $\Pi=\{\alpha_1,\ldots,\alpha_{\ell}\}$ for $\g$,
 and for simplicity, denote $H_i=H_{\alpha_i}$ the simple coroot for $i=1,\ldots,\ell$.
As we need, we recall the definition of a Chevalley basis for a basic classical Lie superalgebra $\g$
(see \cite[Definition 3.2.1]{FG}, etc.).
The existence of a Chevalley basis of $\g$ was first shown in
\cite[Theorem 3.9]{IK} (see also Theorem 3.3.1 of \cite{FG}, etc.).

\begin{de}\label{(de_cheva for basic)}
	{\em Let $\g$ be a basic classical Lie superalgebra
but not type $A(1,1)$ or $D(2,1;a)$ with $a\notin \Z$.
A {\em Chevalley basis} of $\g$ is a homogeneous $\C$-basis
$\{H_1,\ldots,H_\ell\}\sqcup\{X_{\alpha}\;|\;\alpha\in \Delta\}$
such that
		\begin{itemize}
			\item[(1)] $\{H_1,\ldots,H_\ell\}$ is a $\C$-basis of $\h$,
    and $\h_{\Z}:=\operatorname{span}_{\Z}\{H_1,\ldots,H_\ell\}=\operatorname{span}_{\Z}\{H_{\alpha} \ | \ \alpha\in \Delta\}.$
			\item[(2)] $[H_i,H_j]=0,$ $[H_i,X_{\alpha}]=\alpha(H_i)X_{\alpha}$
   for all $1\le i,j\le \ell$ and $\alpha\in \Delta.$
			\item[(3)] $[X_{\alpha},X_{-\alpha}]=\sigma_{\alpha}H_{\alpha}$
for all $\alpha\in \Delta$,
where $\sigma_{\alpha}:=-1$ if $\alpha\in \Delta_{\bar{1}}^{-}$
 (the negative roots in $\Delta_{\bar{1}}$),
 $\sigma_{\alpha}:=1$ otherwise.
		\item[(4)] $[X_{\alpha},X_{\beta}]=c_{\alpha,\beta}X_{\alpha+\beta}$
for $\alpha,\beta\in \Delta$, $\alpha\ne -\beta$, where
			\begin{itemize}
				\item[(4.0)] if $\alpha+\beta\notin \Delta$,
    then $c_{\alpha,\beta}=0$ and $X_{\alpha+\beta}:=0$;
				\item[(4.1)] if $\alpha\in \Delta_{\bar{0}}$
or $\beta \in \Delta_{\bar{0}}$ (we assume that $\alpha \in \Delta_{\bar{0}}$)
 and $\alpha+\beta\in \Delta$, then $c_{\alpha,\beta}=\pm(p+1)$;
				\item[(4.2)] if $\alpha, \beta\in \Delta_{\bar{1}}$ with
$( \alpha,\alpha) \ne0$ or $( \beta,\beta) \ne0$ (we assume $( \alpha,\alpha) \ne0$)
and $\alpha+\beta \in \Delta$, then $c_{\alpha,\beta}=\pm(p+1)$;
				\item[(4.3)]  if $\alpha, \beta\in \Delta_{\bar{1}}$
with $( \alpha,\alpha) =0=(\beta,\beta) $,
 and $\alpha+\beta \in \Delta$,
  then $c_{\alpha,\beta}=\pm\beta(H_{\alpha}).$
			\end{itemize}
			Here $p:=\operatorname{max}\{i\mid \beta -i\alpha \in \Delta\}$.
	\end{itemize}}
\end{de}

\begin{rem}
	{\em A suitably defined Chevalley basis
for the basic classical Lie superalgebra of type $A(1,1)$
was also proved to exist,
but with the root system $\Delta$ replaced by a generalized version $\tilde{\Delta}$
(see \S 6.1 of \cite{FG}).
 A Chevalley basis for the basic classical Lie superalgebra of type
   $D(2,1;a)$ with $a\notin \Z$ exists over the ring $\Z[\alpha]$
   (see \S 3.2 of \cite{Ga}).
   }
\end{rem}

In the rest of this subsection,
we assume that $\g$ is a basic classical Lie superalgebra
but not type $A(1,1)$ or $D(2,1;a)$ with $a\notin \Z$.
Set $\hat{\h}:=(\h\otimes 1)\oplus\C\c$.
We view $\c\in \hat{\h}^*$ by the identification of $\hat{\h}$ and
$\hat{\h}^{*}$.
The roots of $\hat{\g}$ with respect to the Cartan subalgebra
 $\hat{\h}$ have the form $\alpha+i\c $ for $\alpha\in \Delta\cup\{0\}$,
  $i\in \Z$ and $(i,\alpha)\ne (0,0)$ (see Lemma 18.2.3 of \cite{M2012}, etc.).
  The root space decomposition of $\hat{\g}$ can be written as
\begin{equation}\label{eq:dec}
	\hat{\g}=\hat{\h}\oplus \bigoplus\limits_{(i,\alpha)\ne (0,0)}(\g_{\alpha}\otimes t^i).
\end{equation}
Denote $\hat{\g}_{j\c}=\h\otimes t^{j}$ for $0\neq j\in\Z$,
and
$\hat{\g}_{\alpha+i\c}=\g_{\alpha}\otimes t^i$
for $\alpha\in\Delta, i\in\Z$.

Fix a Chevalley basis
$\{H_1,\ldots,H_\ell\}\sqcup\{X_{\alpha}\;|\;\alpha\in \Delta\}$ of $\g$.
Let $(\cdot,\cdot)$ be an
integer-valued bilinear form on the Chevalley basis of $\g$.
Let $I=\{1,\ldots,\ell\}$ and $\hat{I}=(I\times \Z)\cup \{(0,0)\}$ be the index sets,
and denote
\begin{equation}
	\begin{split}
		X_{\alpha}(m) =X_{\alpha}\otimes t^m,\;\;	H_{i}(n)=H_{i}\otimes t^n, \;\;
		H_{0}(0)=- H_{\theta}\otimes 1+\c,
	\end{split}
\end{equation}
where $\alpha\in\Delta$, $m\in \Z$, $(i,n)\in I\times \Z$
and $\theta$ is the highest root in $\Delta$.
Then the set
\begin{equation}\label{eq:cheva}
\mathcal{C}:=\{X_{\alpha}(m)\mid \alpha\in \Delta,m\in \Z\}\sqcup\{H_{i}(n) \;|\; (i,n)\in \hat{I} \}
\end{equation}
is a basis of $\hat{\g}$.
This basis is known as a {\em Chevalley basis} of $\hat{\g}$ (cf. \cite{G}, \cite{M}, etc.).
Let $\hat{\g}_{\Z}$ be the $\Z$-span of the Chevalley basis $\mathcal{C}$
	   of $\hat{\g}$.
	   Then $\hat{\g}_{\Z}$ is an integral form of $\hat{\g}$ as a Lie superalgebra with
	\begin{equation}
		\begin{split}
			&(\hat{\g}_{\Z})_{\bar{1}}=
    \mathrm{span}_{\Z}\{X_{\alpha}(m)\;|\;\alpha\in \Delta_{\bar{1}}, m\in \Z\},\\
			&(\hat{\g}_{\Z})_{\bar{0}}=
\mathrm{span}_{\Z}\{X_{\alpha}(m), H_{i}(n)\;|\;\alpha\in \Delta_{\bar{0}},m\in \Z, (i,n)\in \hat{I}\}.
		\end{split}
	\end{equation}

Now we construct an integral form for the universal enveloping superalgebra
 $U(\hat{\g})$ of the affine Lie superalgebra $\hat{\g}$.
 We first recall some notation.

 In $\C[[x]]$, set
\begin{eqnarray*}
\mathrm{exp}(x) =\sum\limits_{n\ge 0}\frac{x^{n}}{n!},\;\;
\mbox{and}\;\;
\mathrm{log}(1+x)=\sum\limits_{n\ge 1}(-1)^{n-1}\frac{x^n}{n}.
\end{eqnarray*}
Let $\C[x_1,x_2,\ldots]$ be the algebra of polynomials in the
	mutually commutative independent variables
	$x_1$, $x_2,\ldots$ with coefficients in $\C$.
For $s\in \Z_+$, define
$\Lambda_{s}=\Lambda_{s}(x_1,x_2,\ldots,x_s)\in \C[x_1,x_2,\ldots,x_s]$
to be the coefficient of $\zeta^s$
\begin{equation}\label{(eq:2.26)}
	\sum\limits_{s\ge 0}\Lambda_{s}\zeta^s=\mathrm{exp}(\sum\limits_{j\ge 1} \frac{x_j}{j}\zeta^j).
\end{equation}
For example,
$\Lambda_{0}=1$, $\Lambda_{1}(x_1)=x_1$,
$\Lambda_{2}(x_1,x_2)=\frac{x_1^2}{2!}+\frac{x_2}{2}$,
$\Lambda_{3}(x_1,x_2,x_3)=\frac{x_1^3}{3!}+\frac{x_1x_2}{2}+\frac{x_3}{3}$.


 For $H_{i}(jn)\in \mathcal{C}$, $j=1,\ldots,s$,
denote
\begin{eqnarray*}
	&\Lambda_{s}(H_{i}(jn)_{j=1}^{s})= \Lambda_{s}(H_{i}(n),H_{i}(2n),\ldots,H_{i}(sn))
\end{eqnarray*}
which is the coefficient of $\zeta^s$ in the function
$	\mathrm{exp}(\sum\limits_{j\ge 1}\frac{H_{i}(jn)}{j}\zeta^j)$.
In particular,
\begin{equation}
	\begin{split}
		\Lambda_{s}(H_{i}(j\cdot 0)_{j=1}^{s})&= \Lambda_{s}(H_{i}(0),H_{i}(0),\ldots,H_{i}(0))=  \binom{H_{i}(0)+s-1}{s}\\
		&=\frac{(H_{i}(0)+s-1)(H_{i}(0)+s-2)\cdots H_{i}(0)}{s!}
	\end{split}
\end{equation}
for $i=0, 1,\ldots, \ell$ (see \cite[Lemma 4.1.19]{M} or
 \cite[Lemma 3.1.4]{P}, etc.).
  For $X_{\alpha}(m) \in \mathcal{C}$, $s\in \Z_+$,
 denote the divided power
 \begin{equation}
 	X_{\alpha}(m)^{(s)}=\frac{X_{\alpha}(m)^s}{s!}.
 \end{equation}
Then $X_{\alpha}(m)^{(0)}=1$ is the unit,
 and $X_{\alpha}(m)^{(1)}=X_{\alpha}(m)$.

Fix a total order $\preceq$ on the Chevalley basis $\mathcal{C}$ (\ref{eq:cheva}) such that
$X_{\beta}(m) \preceq H_{i}(m')\preceq X_{\gamma}(m'')$
for $\beta\in \Delta_{\bar{0}}$, $\gamma\in \Delta_{\bar{1}}$,
$(i,m')\in \hat{I}$, $m, m''\in \Z$.
Let $K_{\Z}(\hat{\g})$ be the $\Z$-subalgebra of $U(\hat{\g})$ generated by the elements
\begin{eqnarray}\label{eq:gen}
{X_{\beta}(m)^{(r)}},\; \Lambda_{s}(H_{i}(jm')_{j=1}^{s}), \, X_{\gamma}(m'')
\end{eqnarray}
for $\beta \in \Delta_{\bar{0}}$, $\gamma\in \Delta_{\bar{1}}$,
$(i,m')\in \hat{I}$,
  $m,m''\in \Z$, and $r,s\in \Z_+$.
Denote by
\begin{equation}\label{eq:genpol}
	\begin{split}
		&f_{\beta}=X_{\beta}(m_1)^{(s_{\beta,1})}\cdots X_{\beta}(m_p)^{(s_{\beta,p})};\\
	&f_{H_i}=\Lambda_{s_{i,1}}(H_{i}(jm_1')_{j=1}^{s_{i,1}})\cdots\Lambda_{s_{i,k}}(H_{i}(jm_k')_{j=1}^{s_{i,k}});\\
	&f_{\gamma}=X_{\gamma}(m_1'')^{s_{\gamma,1}}\cdots X_{\gamma}(m_q'')^{s_{\gamma,q}}
	\end{split}
\end{equation}
for $\beta\in \Delta_{\bar{0}}$, $\gamma \in \Delta_{\bar{1}}$,
 $(i,m_1'),\ldots, (i,m_k') \in \hat{I}$, $m_1\le \ldots \le m_p$,
$m_1'\le \ldots \le m_k'$,  $m_1'' \le \ldots \le m_q''$,
 $p,q,k,s_{\beta,l}$, $s_{i,l}\in \Z_+$, and $s_{\gamma,l}\in \{0,1\}$.
We call
{\em monomials} the products of the elements of the form
\begin{eqnarray}\label{monomials}
	\prod_{\beta\in \Delta_{\bar{0}}}f_{\beta}, \,\, \prod_{i\in I\cup \{0\}}f_{H_i},\,\,\text{or} \,\,\prod_{\gamma\in\Delta_{\bar{1}}}f_{\gamma}.
\end{eqnarray}

We identify $\hat{\g}$ with its image in
$U(\hat{\g})$ by the PBW theorem.
 Consider the canonical filtration of $U(\hat{\g})$ (cf. Section 6.3 of \cite{M2012})
\begin{eqnarray*}
\C=U(\hat{\g})_{0}\subseteq U(\hat{\g})_{1}\subseteq \ldots,
\end{eqnarray*}
where
$U(\hat{\g})_{1}=\C+\hat{\g}$,
and $U(\hat{\g})_{n}=(U(\hat{\g})_{1})^{n}$ for $n\in\N$;
that is, $U(\hat{\g})_{n}$ is the span of all products
$u_{1}u_{2}\cdots u_{n}$ with $u_{i}\in U(\hat{\g})_{1}$.
For $u\in U({\hat{\g}})$, we say the {\em degree} of $u$,
denoted by $\mathrm{deg} \,u$, to be the smallest integer $j$ such that $u\in U({\hat{\g}})_{j}$.

\begin{lem}\label{(key Lem)}
	\begin{itemize}
		\item[(1)] Let $C_1$ and $C_2$ be any two elements listed in (\ref{eq:gen}).
 Then
		\begin{equation}
			C_2C_1=\pm C_1C_2+P,
		\end{equation}
		where $P$ is a $\Z$-linear combination of products $Q$ of monomials in (\ref{monomials})
such that $\mathrm{deg}(Q)<\mathrm{deg}(C_1)+\mathrm{deg}(C_2)$.
		\item[(2)] For $r,s\in \Z_{+}$,
		$\beta\in \Delta_{\bar{0}}$, $(i,n)\in\hat{I}$,
		we have
		\begin{equation}
			X_{\beta}(m)^{(r)}	X_{\beta}(m)^{(s)}=\binom{r+s}{s}X_{\beta}(m)^{(r+s)}+P_1,
		\end{equation}
		\begin{equation}\label{(eq:2.42)}
			\Lambda_{r}(H_{i}(jn)_{j=1}^{r})\Lambda_{s}(H_{i}(jn)_{j=1}^{s})=\binom{r+s}{s}	\Lambda_{r+s}(H_{i}(jn)_{j=1}^{r+s})+P_2,
		\end{equation}
		where $P_1$, $P_2$ are $\Z$-linear combinations products $Q$ of monomials in (\ref{monomials}) such that
		$\mathrm{deg}(Q)<r+s.$
	\end{itemize}
\end{lem}

\begin{proof}
	The conclusion (2) follows from Lemma 4.2.13 (ii) of \cite{M}
	since
$ X_{\beta}(m),\,\, H_{i}(n)\in \hat{\g}_{\bar{0}}=\g_{\bar{0}}\otimes \C[t,t^{-1}]\oplus \C\c $
	for $\beta \in \Delta_{\bar{0}}$, $(i,n)\in\hat{I}$, $m\in\Z$,
and $\g_{\bar{0}}$ is a finite-dimensional simple Lie algebra.

Now we prove (1).
It can be directly checked that for the following $C_1$ and $C_2$,
 either $C_2C_1 = C_1C_2$ or $C_2C_1 = -C_1C_2$:
\begin{itemize}
	\item[(i)]\label{(1)}
	$C_1=X_{\alpha}(m)^{(r)}$, $C_{2}=X_{\beta}(n)^{(s)}$ for
$m,n\in\Z$, $\alpha, \beta\in \Delta$ with $r$ or $s=0$,
or $r,s\in \N$ and $\alpha+\beta \notin \Delta\cup\{0\}$.
	\item[(ii)]
	$C_1=\Lambda_{r}(H_{i}(kn)_{k=1}^{r})$,
	$C_{2}=\Lambda_{s}(H_{j}(lm)_{l=1}^{s})$ for $(i,n)$, $(j,m)\in \hat{I}$ ,
$r$ or $s=0$,
or $r,s\in \N$
with $mn\ge 0$.
\end{itemize}

Let $r,s\in \N$, $m,n\in \Z$.
It remains to show the following six cases.
	\begin{itemize}
		\item[{\em Case 1.}] $C_1=X_{\alpha}(m)$, $C_2=X_{\beta}(n)$
with $\alpha$, $\beta \in \Delta_{\bar{1}}$, $\alpha+\beta\in\Delta\cup\{0\}$;
		\item[{\em Case 2.}] $C_1=X_{\alpha}(m)^{(r)}$, $C_2=X_{\beta}(n)$
with $\alpha \in \Delta_{\bar{0}}$, $\beta \in \Delta_{\bar{1}}$, $\alpha+\beta\in\Delta$;
		\item[{\em Case 3.}] $C_1=X_{\alpha}(m)^{(r)}$, $C_2=X_{\beta}(n)^{(s)}$
with  $\alpha$, $\beta \in \Delta_{\bar{0}}$, $\alpha+\beta\in\Delta\cup\{0\}$.
		\item[{\em Case 4.}] $C_1=X_{\alpha}(m)$,
		$C_2=\Lambda_{s}(H_{i}(jn)_{j=1}^{s})$
		with $\alpha\in\Delta_{\bar{1}}$, $(i,n)\in \hat{I}$;
		\item[{\em Case 5.}] $C_1=X_{\alpha}(m)^{(r)}$,
		$C_2=\Lambda_{s}(H_{i}(jn)_{j=1}^{s})$
with $\alpha\in\Delta_{\bar{0}}$, $(i,n)\in \hat{I}$.
	\item[{\em Case 6.}] $C_1=\Lambda_{r}(H_{i}(kn)_{k=1}^{r}), C_1=\Lambda_{s}(H_{j}(lm)_{l=1}^{s})$
 for $r,s\in \N$,  $(i,n), (j,m)\in \hat{I}$ with $mn<0$.
\end{itemize}

For {\em Case 1},
if $\alpha+\beta \ne 0$, we have
\begin{equation}
	\begin{split}
		X_{\beta}(n)X_{\alpha}(m)
		&=-X_{\alpha}(m)X_{\beta}(n)+[X_{\beta}(n),X_{\alpha}(m)]\\
		&=-X_{\alpha}(m)X_{\beta}(n)+c_{\beta,\alpha}X_{\alpha+\beta}(m+n).
	\end{split}
\end{equation}
So that $P=c_{\beta,\alpha}X_{\alpha+\beta}(m+n)$,
$Q=X_{\alpha+\beta}(m+n)$ and
$\mathrm{deg}(Q)=1<
\mathrm{deg}(X_{\beta}(n))+\mathrm{deg}(X_{\alpha}(m))=2$.
	If $\alpha+\beta=0$, then
\begin{equation}
	\begin{split}
		X_{-\alpha}(n)X_{\alpha}(m)
		&=-X_{\alpha}(m)X_{-\alpha}(n)+[X_{-\alpha}(n),X_{\alpha}(m)]\\
		&=-X_{\alpha}(m)X_{-\alpha}(n)+\sigma_{\alpha}H_{\alpha}(m+n)+n\delta_{m+n,0}(X_{-\alpha},X_{\alpha})\c.
	\end{split}
\end{equation}
Set $P=\sigma_{\alpha}H_{\alpha}(m+n)+n\delta_{m+n,0}(X_{-\alpha},X_{\alpha})\c$.
Then $P$
 is an integral linear span of $Q=H_{i}(m+n)$ and $H_{j}(0)$
 for $(i,m+n)$, $(j,0)\in \hat{I}$, and $\mathrm{deg}(Q)=1<2$.
Note that $\Lambda_1(H_{i}(km)_{k=1}^{1})=\Lambda_1(H_{i}(m))=H_{i}(m)$ for $(i,m)\in \hat{I}$.

For {\em Case 2},
by Proposition 2.5.5 of \cite{Kac} or Proposition 1.3 of \cite{Kac2},
we can further divide it into two cases: $\alpha\ne c\beta$ for any $c\in\C$ or $\alpha=\pm 2\beta$.
If $\alpha=2\beta$, then $\alpha+\beta\notin \Delta$,
 and this is already considered in (i).
 	If $\alpha=-2\beta$,
 	then by
 	Definition \ref{(de_cheva for basic)},
 	$[X_{\beta}(n),X_{\alpha}(m)]
 	=\pm X_{\alpha+\beta}(m+n)=\pm X_{-\beta}(m+n)$.	
 By induction on $r$, we have
 \begin{equation}\label{eq:2.39}
 	X_{\beta}(n)X_{\alpha}(m)^{(r)}=X_{\alpha}(m)^{(r)}X_{\beta}(n)+c_{\beta,\alpha}X_{\alpha}(m)^{(r-1)}X_{-\beta}(m+n),
 \end{equation}
where $c_{\beta,\alpha}=\pm 1$.
 	If $\alpha\ne c\beta$,
 	 by Definition \ref{(de_cheva for basic)} and induction on $r$, we have
 \begin{eqnarray}\label{eq:2.37}
 	\begin{aligned}
 	&X_{\beta}(n)X_{\alpha}(m)^{(r)}=X_{\alpha}(m)^{(r)}X_{\beta}(n) \\
 	&\;\;\;\;+\sum\limits_{k=1}^{r}\left(\prod_{\ell=1}^{k} \epsilon_{\ell} \right)\binom{p+k}{k}X_{\alpha}(m)^{(r-k)} X_{\beta+k\alpha}(n+km),
\end{aligned}
\end{eqnarray}
 where $p=\mathrm{max}\{i \;|\; \beta-i\alpha \in \Delta\}$,
 $X_{\beta+k\alpha}(n+km):=0$ if $\beta+k\alpha \notin \Delta$,
 and $\epsilon_{\ell}\in\{\pm 1\}$ is such that
 $
 	[X_{\alpha}(m), X_{\beta+(\ell-1)\alpha}(n+\ell m-m)]=\epsilon_{\ell}(p+\ell) X_{\beta+\ell \alpha}(n+\ell m).
$
 Both the equations (\ref{eq:2.37}) and (\ref{eq:2.39})
show that the result holds for {\em Case 2}.

{\em Case 3},
{\em Case 5} and {\em Case 6} have been proved in Lemma 4.2.13 of
\cite{M}.
At last, we study {\em Case 4}.
For $\alpha\in \Delta_{\bar{1}}$,
we have $[H_{i}(n),X_{\alpha}(m)]=\alpha(H_{i})X_{\alpha}(m+n)$,
which is similar for $\alpha\in \Delta_{\bar{0}}$.
So we can get a similar result of Lemma 4.3.4 (iii) of \cite{M} for $\alpha\in \Delta_{\bar{1}}$.
Then by considering the coefficient of $\zeta_1^1\zeta_2^s$ in Lemma 4.3.4 (iii) of \cite{M},
we have
\begin{equation}\label{eq:cs4}
	\begin{aligned}
		\Lambda_{s}((H_{i}(jn))_{j=1}^{s})X_{\alpha}(m)
		&=X_{\alpha}(m)\Lambda_{s}((H_{i}(jn))_{j=1}^{s})\\
		&+\sum\limits_{k=1}^{s}\binom{\alpha(H_i)+k-1}{k}
		X_{\alpha}(m+kn)\Lambda_{s-k}((H_{i}(jn))_{j=1}^{s-k})
	\end{aligned}
\end{equation}
for $(i,n)\in I\times \Z$,
and for $(i,n)=(0,0)$,
\begin{eqnarray}\label{eq:cs4-2}
\begin{aligned}
&\binom{H_{0}(0)+\c+s-1}{s}X_{\alpha}(m)
=X_{\alpha}(m)\binom{H_{0}(0)+\c+s-1}{s} \\
&\quad+X_{\alpha}(m)\sum_{k=1}^{s}\binom{\alpha(H_0)+k-1}{k}\binom{H_{0}(0)+\c+s-k-1}{s-k}.
\end{aligned}
\end{eqnarray}
It follows from (\ref{eq:cs4}) and (\ref{eq:cs4-2}) that the result holds for {\em Case 4}.

\end{proof}

We now show that
 $K_{\Z}(\hat{\g})$ (see \eqref{eq:gen}) is an integral form of $U(\hat{\g})$.

\begin{thm}\label{(th 2.13)}
	Let $\g$ be a basic classical Lie superalgebra,
excluding types $A(1,1)$ and $D(2,1;a)$ with $a\notin \Z$.
Then
	the set of ordered monomials of the form
	\begin{equation}\label{PBW like basis}
		\prod_{\beta\in \Delta_{\bar{0}}}f_{\beta}\prod_{i\in I\cup \{0\}}f_{H_i} \,\,\prod_{\gamma\in\Delta_{\bar{1}}}f_{\gamma},\;\;
	\end{equation}
forms an integral basis of $K_{\Z}(\hat{\g})$,
 where $f_{\beta}$, $f_{H_{i}}$ and $f_{\gamma}$ are defined in (\ref{eq:genpol}).
 So that $K_{\Z}(\hat{\g})$ is an integral form of $U(\hat{\g})$, where
	$K_{\Z}(\hat{\g})_{\bar{i}}$ is the $\Z$-span of
(ordered) products (\ref{PBW like basis}) with
$\sum\limits_{\gamma\in \Delta_{\bar{1}}}( s_{\gamma,1}+\ldots+ s_{\gamma,q})\equiv i \pmod 2$
 for $i \in \{0,1\}$.
\end{thm}
\begin{proof}
Let $U^{\prime}_{\Z}(\hat{\g})$ be the $\Z$-subalgebra of $U(\hat{\g})$ with basis (\ref{PBW like basis}).
Clearly, $K_{\Z}(\hat{\g})\subseteq U^{\prime}_{\Z}(\hat{\g})$.
	 Lemma \ref{(key Lem)} implies that any product $Q$ of elements in (\ref{eq:gen}) is
 a $\Z$-linear combination of the (ordered) products of the form (\ref{PBW like basis}).
 So we get the reverse inclusion, and hence
 the set of ordered products (\ref{PBW like basis}) forms a $\Z$-basis of $K_{\Z}(\hat{\g})$.

By the PBW theorem, the set of ordered products
(\ref{PBW like basis}) forms a $\C$-basis of $U(\hat{\g})$, i.e.
	$K_{\Z}(\hat{\g})\otimes_{\Z}\C=U(\hat{\g})$.
For $i\in\{0,1\}$,
$U(\hat{\g})_{\bar{i}}$ can be linearly spanned by the (ordered) products in
(\ref{PBW like basis}) with
$\sum\limits_{\gamma\in \Delta_{\bar{1}}}( s_{\gamma,1}+\ldots+ s_{\gamma,q})\equiv i \pmod 2$.
	We have
	$
		K_{\Z}(\hat{\g})_{\bar{i}}=K_{\Z}(\hat{\g})\cap U(\hat{\g})_{\bar{i}}.
$
Therefore, $K_{\Z}(\hat{\g})$ is an integral form of $U(\hat{\g})$.
\end{proof}

There are some differences between the generators of $K_{\Z}(\hat{\g})$
for $\g$ a basic classical Lie superalgebra and
 the generators of $K_{\Z}(\hat{\g})$ (constructed in \cite{M})
  for $\g$ a finite-dimensional simple Lie algebra.
  Let $\alpha_{i}\in\Delta$,
  $H_{i}$ the corresponding coroot, $s\geq 1$, $n\in\Z$.
It follows from the proof of Theorem 4.2.6 of \cite{M} that
for $\alpha_i$ an even root,
 $\Lambda_{s}(H_{i}(jn)_{j=1}^{s})$ can be generated over $\Z$ by
  $X_{\alpha}(m)^{(r)}$
  for $\alpha\in\Delta$, $r\geq 1$, $m\in\Z$.
However,
 if $\alpha_{i}\in\Delta_{\bar{1}}$,
 we can't show that $\Lambda_{s}(H_{i}(jn)_{j=1}^{s})$
can be generated by $X_{\alpha}(m)^{(r)}$ over $\Z$
for $\alpha\in \Delta(\g)$, $r\ge 1$, $m\in\Z$.
Therefore,
 we instead consider $s!\Lambda_{s}(H_{i}(jn)_{j=1}^{s})$ for $\alpha_i\in \Delta_{\bar{1}}$.
  It is evident that $s!\Lambda_{s}(H_{i}(jn)_{j=1}^{s})$ can be generated over $\Z$ by
   $X_{\gamma}(m^{\prime})$ for some $\gamma\in \Delta_{\bar{1}}$, $m^{\prime}\in\Z$.
   For example, we can write
   $H_{i}(jn)=\sigma_{\alpha_{i}}(X_{\alpha_{i}}(0)X_{-\alpha_{i}}(jn)-X_{-\alpha_{i}}(jn)X_{\alpha_{i}}(0))$
   for $j\geq 1$, $n\in\Z$, where $\sigma_{\alpha_{i}}\in\{\pm 1\}$.
This motivates us to give another integral form of $U(\hat{\g})$,
which is closely related to the integral form $K_{\Z}(\hat{\g})$.

For $(i,n)\in \hat{I}$, denote
	\begin{equation}\label{eq:tildH}
		\widetilde{\Lambda_s}(H_i(jn)_{j=1}^{s})=
		\begin{cases}
		\Lambda_{s}(H_{i}(jn)_{j=1}^{s}), \;\;\;\;\;\text{if}\; \alpha_i \in\Delta_{\bar{0}} \\
		s!\Lambda_{s}(H_{i}(jn)_{j=1}^{s}),\;\;\text{if}\; \alpha_i \in \Delta_{\bar{1}}.
		\end{cases}
	\end{equation}
	Let $\tilde{U}_{\Z}(\hat{\g})$ be the $\Z$-subalgebra of $U(\hat{\g})$ generated by elements
	\begin{eqnarray*}
X_{\beta}(m)^{(s)},\;\; X_{\gamma}(m'),\,\;
		\mbox{and} \; \widetilde{\Lambda_s}(H_i(jn)_{j=1}^{s})
\end{eqnarray*}
	for $\beta \in \Delta_{\bar{0}}$, $\gamma\in \Delta_{\bar{1}}$,
$(i,n)\in \hat{I}$,   $m,m'\in \Z$, and $s\in \Z_+$.
Clearly, $\tilde{U}_{\Z}(\hat{\g})\subset K_{\Z}(\hat{\g})$.

Note that a similar result of Lemma \ref{(key Lem)} still holds,
except that (\ref{(eq:2.42)}) is adjusted to
\begin{equation}	
	\widetilde{\Lambda_{r}}(H_{i}(jn)_{j=1}^{r})\widetilde{\Lambda_{s}}(H_{i}(jn)_{j=1}^{s})
=\widetilde{\Lambda_{r+s}}(H_{i}(jn)_{j=1}^{r+s})+P',
\end{equation}
where $P'$ is a $\Z$-linear combination of (ordered) products $Q'$ of the elements of the form
$\widetilde{\Lambda_{s}}(H_{i}(jn)_{j=1}^{s})$
such that
$\mathrm{deg}(Q')<r+s.$
Denote
\begin{equation}
\widetilde{f_{H_i}}=
\widetilde{\Lambda_{s_{i,1}}}(H_{i}(jm_1')_{j=1}^{s_{i,1}})\cdots \widetilde{\Lambda_{s_{i,k}}}(H_{i}(jm_k')_{j=1}^{s_{i,k}})
\end{equation}
for $(i,m_j')\in \hat{I},$ $m_1'\ge \dots \ge m_k'$, and $s_{i,j}\in \Z_+$.
Similar to Theorem \ref{(th 2.13)}, we have the following result.

\begin{thm}\label{th:anotherform)}
	Let $\g$ be a basic classical Lie superalgebra,
excluding types $A(1,1)$ and $D(2,1;a)$ with $a\notin \Z$.
Then
$\tilde{U}_{\Z}(\hat{\g})$ is the same as the $\Z$-subalgebra generated by
	\begin{equation*}
		X_{\beta}(m)^{(s)},\;\; X_{\gamma}(m')
	\end{equation*}
	 for $\beta \in \Delta_{\bar{0}}$, $\gamma\in \Delta_{\bar{1}}$,  $m,m'\in \Z$, and $s\in \Z_+$.
	The set of ordered monomials of the form
	\begin{equation}
	\prod_{\beta\in \Delta_{\bar{0}}} f_{\beta} \prod_{i\in I\cup \{0\}}\widetilde{f_{H_i}}\prod_{\gamma\in\Delta_{\bar{1}}}f_{\gamma},
\end{equation}
forms an integral basis of $\tilde{U}_{\Z}(\hat{\g})$.
 And $\tilde{U}_{\Z}(\hat{\g})$ is an integral form of
 $U(\hat{\g})$ which satisfies (\ref{eq3.12}) and (\ref{eq:2.15}).
\end{thm}

\section{Integral forms for affine vertex operator superalgebras and their modules}\label{sec:3}

In this section, we first extend the notion of
integral forms from vertex operator algebras to vertex operator superalgebras and their modules.
Then we study integral forms for affine vertex operator superalgebras
and their modules.
We will construct natural integral forms in
vertex operator superalgebras and their modules
based on a basic classical Lie superalgebra $\g$,
using the integral forms $\tilde{U}_{\Z}(\hat{\g})$
constructed in the previous section,
and we find natural generating sets for these integral forms.
Moreover, we give a criterion for when an integral form
of a vertex operator superalgebra contains its conformal vector.
Besides,
integral forms in contragredient modules of vertex operator superalgebras
are also studied.

\subsection{Basic notions and notations}

We follow the setting of \cite{L} for the definitions of
vertex (operator) superalgebra (\cite[Definition 2.2.1]{L})
 and their modules (\cite[Definition 2.3.1]{L}).
Let $V$ be a vertex (operator) superalgebra with $\Z_{2}$-graded
  super vertex operator $Y$ over $\C$.
 For any $a\in V$, $n\in\Z$, we denote by $a(n)\in\mbox{End}(V)$
 the associated vertex super operator.
 If $\omega$ is the conformal vector of $V$,
 we denote $L(n)=\omega(n+1)$ for $n\in\Z$.
Note that all numerical coefficients in the formal delta functions appearing
in the Jacobi identity of the definition of a vertex superalgebra are integers.
 Hence the notion of vertex superalgebra over $\Z$ makes sense.

Now we give the definitions of integral forms for
vertex (operator) superalgebras and their modules.
These generalize the notions of integral forms of vertex algebras
 and their modules (see Section 2 in \cite{M1},
 or Definitions 2.1 and 2.4 of \cite{M2}, etc.).

 \begin{de}\label{(de:int of VSA)}
	{\em
\par(1)
Let $V$ be a vertex superalgebra.
		 An {\em integral form} of $V$ is a vertex subsuperalgebra over $\Z$
$V_{\Z}\subseteq V$ which is an integral form of $V$ as a super vector space.
\par(2)
 If further $V$ is a vertex operator superalgebra
		with the conformal vector $\omega$,
then an {\em integral form of $V$} is an integral form $V_{\Z}$ of $V$ as a
		vertex superalgebra and is compatible with the conformal weight grading of $V$:
		\begin{equation}
			V_{\Z}=\coprod_{n\in\Z}V_{(n)}\cap V_{\Z},
\;\;\mbox{where}\;V_{(n)}=\{v\in V \; |\; L(0)v=nv\}.
		\end{equation}		
\par(3)
	Let $W$ be a vertex superalgebra $V$-module,
and $V_{\Z}$ an integral form of $V$.
An {\em integral form in a $V$-module $W$} is
a $V_{\Z}$-submodule $W_{\Z}\subseteq W$ that is an integral form of $W$ as a super vector space.
\par(4)
If further $V$ is a vertex operator superalgebra and $W$ is a $V$-module,
then an {\em integral form in $W$} is
		an integral form $W_{\Z}$ of $W$ as a vertex superalgebra $V$-module
		 and is compatible with the conformal weight grading of $W$:
		\begin{equation}
			W_{\Z}=\coprod_{h\in \C}W_{(h)}\cap W_{\Z},
\;\;\mbox{where}\;W_{(h)}=\{w\in W \; |\; L(0)w=hw\}.
		\end{equation}
	}
\end{de}

\begin{rem}
	{\em	In other words,
		an integral form $V_{\Z}$ of a vertex superalgebra $V$ is the $\Z$-span of a basis for $V$
		which contains the vacuum ${\bf 1}$, is closed under vertex superalgebra products,
and is compatible with the $\Z_2$-grading of $V$.
		Therefore, we have $\mathbf{1} \in (V_{\Z})_{\bar{0}}$,
and $a(n)b \in (V_{\Z})_{\lambda + \mu}$ for any $a \in (V_{\Z})_{\lambda}$,
$b \in (V_{\Z})_{\mu}$, $\lambda, \mu \in \Z_2$, and $n \in \Z$.	
An integral form $V_{\Z}$ of a vertex operator superalgebra $V$
		is further compatible with the $\Z_{2}\times\Z$-gradation
		\begin{equation}
		V_{\Z}=\coprod_{\lambda\in\Z_{2}, n\in\Z}V_{(n)}^{\lambda}\cap V_{\Z},
\;\;\mbox{where}\;V_{(n)}^{\lambda}=V_{(n)}\cap V_{\lambda}.
	\end{equation}
		
		Similarly, an integral form $W_{\Z}$ of a vertex superalgebra $V$-module $W$ is
 the $\Z$-span of a basis for $W$ which is preserved by vertex super operators from $V_{\Z}$,
		and is compatible with the $\Z_2$-grading of $W$.
		Hence, $a(n)w\in (W_{\Z})_{\lambda+\mu}$ for any $a\in(V_{\Z})_{\lambda}$,
$w\in(W_{\Z})_{\mu}$, $n\in \Z$, $\lambda,\mu \in \Z_2$.
An integral form $W_{\Z}$ of a vertex operator superalgebra $V$-module $W$
	is further compatible with the $\Z_{2}\times\C$-gradation
	\begin{equation}
		W_{\Z}=\coprod_{\lambda\in\Z_{2}, h\in\C}W_{(h)}^{\lambda}\cap W_{\Z},
\;\;\mbox{where}\;W_{(h)}^{\lambda}=W_{(h)}\cap W_{\lambda}.
	\end{equation}
}	
\end{rem}

Similar to Proposition 2.3 in \cite{M1},
if $V$ is a vertex superalgebra with an integral form $V_{\Z}$,
then $V_{\Z}\cap \C\mathbf{1}=\Z\mathbf{1}$.

Let $S$ be a subset of a vertex superalgebra $V$ which
consists of homogeneous elements.
We define $\left\langle S\right\rangle_{\Z}$ to be the
 smallest vertex subsuperalgebra over $\Z$ containing $S$,
  and we call $\left\langle S\right\rangle_{\Z}$ the
{\em vertex subsuperalgebra over $\Z$ generated by $S$}.
Then $\left\langle S\right\rangle_{\Z}$ is the intersection of all
vertex subsuperalgebra over $\Z$ of $V$ containing $S$.		
		
The proof of the following useful general result
on vertex subsuperalgebras is similar to the
proof of Proposition 3.9.3 in \cite{LL}.
\begin{prop}\label{(prop 2.7)}
Let $V$ be a vertex superalgebra. For a subset $S$ of $V$
consisting of homogeneous elements,
the vertex subsuperalgebra over $\Z$ $\left\langle S\right\rangle_{\Z}$ is
 the $\Z$-span of the coefficients of the products in the expression:
\begin{equation}\label{eq:3.5}
	Y(u^1,x_1)\cdots Y(u^k,x_k)\mathbf{1},
\end{equation}
where $u^1,\ldots,u^k\in S$, $k\in \Z_+$,
and the $\Z_{2}$-grading of the element (\ref{eq:3.5})
is $(|u^{1}|+\cdots+|u^{k}|) (\mathrm{mod}\;2)$.
Moreover, if $W$ is a $V$-module
and $T$ is a subset of $W$ consisting of homogeneous elements,
the $\left\langle S\right\rangle_{\Z}$-submodule generated by $T$ is the $\Z$-span of coefficients of the form
 \begin{equation}\label{eq:3.6}
 	Y(u^1,x_1)\cdots Y(u^k,x_k)w,
 \end{equation}
 where $u^1,\ldots,u^k\in S$, $w\in T$,
 and the $\Z_{2}$-grading of the element (\ref{eq:3.6})
 is $(|u^{1}|+\cdots+|u^{k}|+|w|) (\mathrm{mod}\;2)$.
\end{prop}

\subsection{Integral forms for affine vertex operator superalgebras and their modules}

Let $\g$ be a finite-dimensional simple Lie superalgebra equipped with
a non-degenerate even supersymmetric invariant bilinear form $(\cdot,\cdot)$.
Let $\hat{\g}$ be the associated affine Lie superalgebra.
For $\ell\in \C$,
$U=U_{\bar{0}}\oplus U_{\bar{1}}$ a finite-dimensional $\g$-module,
let $V_{\hat{\g}}(\ell,0)$ be the associated affine vertex superalgebra,
and $V_{\hat{\g}}(\ell,U)$ a $V_{\hat{\g}}(\ell,0)$-module
 (cf. \cite{KW}, \cite{L}, \cite{Xu}, \cite{Z}, etc.).
 Denote by $L_{\hat{\g}}(\ell,0)$ the corresponding simple affine vertex superalgebra.
 The module $V_{\hat{\g}}(\ell,U)$ has a unique irreducible quotient,
 denoted by $L_{\hat{\g}}(\ell,U)$.
 If $\ell\ne -h^{\vee}$, where $h^{\vee}$ is the dual Coxeter number of $\g$,
then $V_{\hat{\g}}(\ell,0)$ is a $\Z$-graded vertex operator superalgebra of central charge
$
\frac{\ell\cdot\mathrm{sdim}\g}{\ell+h^{\vee}},
$
where $\mathrm{sdim}\g$ is the super dimension of the Lie superalgebra $\g$.

\begin{thm}\label{(th3.3)}
Assume that the pair $(\g, (\cdot,\cdot))$ has an integral form ${\g}_{\Z}$.
	Suppose $U_{\Z}=(U_{\Z})_{\bar{0}}\oplus (U_{\Z})_{\bar{1}}$ is
an integral form of a finite-dimensional $\hat{\g}_{(0)}=\g\oplus\C\c$-module $U$,
where $\c$ acts as
	a scalar $\ell\in \Z$.
	Let $W$ be either $V_{\hat{\g}}(\ell,U)$ or $L_{\hat{\g}}(\ell,U)$.
	Then $W_{\Z}=U_{\Z}(\hat{\g})U_{\Z}$ is an integral form of $W$
as a super vector space with the $\Z_2$-grading
 \begin{eqnarray}\label{eq:grad}
(W_{\Z})_{\bar{i}}=\sum\limits_{j+k\equiv i\pmod{2}}U_{\Z}(\hat{\g})_{\bar{j}}(U_{\Z})_{\bar{k}},
\end{eqnarray}
	where $i\in\{0,1\}$,
$U_{\Z}(\hat{\g})$ is constructed in Proposition \ref{prop:2.5}.
	Moreover, $W_{\Z}$ is compatible with the conformal weight grading of $W$,
	 and is invariant under $U_{\Z}(\hat{\g})$.
\end{thm}

\begin{proof}
It is clear that $W_{\Z}$ is invariant under $U_{\Z}(\hat{\g})$.
Note that $\hat{\g}_{(+)}.U=0$ and
 $U_{\Z}(\hat{\g}_{(0)})U_{\Z}= U_{\Z}$,
 we get
	\begin{equation}\label{eq:3.13}
		W_{\Z}=	U_{\Z}(\hat{\g})U_{\Z}
=U_{\Z}(\hat{\g}_{(-)})U_{\Z}(\hat{\g}_{(0)})U_{\Z}(\hat{\g}_{(+)})U_{\Z}
=U_{\Z}(\hat{\g}_{(-)})U_{\Z}.
	\end{equation}
By (\ref{eq:2.15}),$
		W_{\Z}=\coprod_{m \in \Z}( U_{\Z}(\hat{\g}_{(-)})\cap U(\hat{\g}_{(-)})_{(m)})U_{\Z}
$
	is graded by conformal weight
and
	the intersection of $W_{\Z}$ with each weight space is spanned by finitely many vectors.
	It can be proved similarly to Theorem 3.2 of \cite{M1}
	that $W$ is linearly isomorphic to $W_{\Z}\otimes_{\Z}\C$.

There is a natural $\Z_{2}$-grading on $W_{\Z}$ given by (\ref{eq:grad}).
It remains to show that $W_{\Z}$ is compatible
with the $\Z_{2}$-grading of $W$.
	By Proposition \ref{prop:2.5}
	and the assumption that $\c$ acts as $\ell\in \Z$,
	we see that $W_{\Z}$ is the $\Z$-linear combination of elements
	of the form
	\begin{equation}\label{(eq:3.12)}
		a_{j_1}(m_{1})^{k_1}\cdots a_{j_l}(m_{l})^{k_l}b_{j^{\prime}_1}(n_{1})\cdots b_{j^{\prime}_p}(n_{p})u
	\end{equation}
	for $a_{j}(m)\in(\hat{\g}_{\Z})_{\bar{0}}$,
	$b_{j^{\prime}}(n)\in (\hat{\g}_{\Z})_{\bar{1}}$,
	 $u\in (U_{\Z})_{\bar{0}}\cup (U_{\Z})_{\bar{1}}$,
	 $k_q\in \N$.
	For $i\in \{0,1\}$,
 $(W_{\Z})_{\bar{i}}$ is the $\Z$-span of the elements in
 (\ref{(eq:3.12)}) such that $p+|u|\equiv i\,(\mathrm{mod}\,2)$.
 Note that by the construction of $W$ (i.e. $V_{\hat{\g}}(\ell,U)$
  or $L_{\hat{\g}}(\ell,U)$, see, for example, \cite{Z}, etc.),
	the elements in (\ref{(eq:3.12)}) also span $W_{\bar{i}}$ over $\C$.
	Therefore, $(W_{\Z})_{\bar{i}}= W_{\bar{i}}\cap W_{\Z}$.
\end{proof}

We now show a general result on vertex superalgebraic integral forms
that applies to any finite-dimensional simple Lie superalgebra $\g$ having an
integral form.

\begin{prop}\label{(prop3.4)}
	Suppose $\g_{\Z}$ is an integral form of $\g$,
	 $U_{\Z}(\hat{\g})$ is constructed in Proposition \ref{prop:2.5} and $\ell\in \Z$.
	 Then the integral form $V_{\hat{\g}}(\ell,0)_{\Z}$ given by
Theorem \ref{(th3.3)} is the vertex subsuperalgebra over $\Z$ generated by
	 the vectors $a(-1)\mathbf{1}$ for
	  $a\in(\g_{\Z})_{\bar{0}}\cup (\g_{\Z})_{\bar{1}}$.
	  Moreover, if $U=U_{\bar{0}}\oplus U_{\bar{1}}$ is a
finite-dimensional $\g$-module with integral form $U_{\Z}$,
	   then $V_{\hat{\g}}(\ell,U)_{\Z}$ and $L_{\hat{\g}}(\ell,U)_{\Z}$
 are the $V_{\hat{\g}}(\ell,0)_{\Z}$-modules generated by $U_{\Z}$.
\end{prop}

\begin{proof}
	Let $W$ be the module $V_{\hat{\g}}(\ell,U)$ or $L_{\hat{\g}}(\ell,U)$.
Note that for $U=\C_{\ell}$, $V_{\hat{\g}}(\ell,U)=V_{\hat{\g}}(\ell,0)$
 and $L_{\hat{\g}}(\ell,U)=L_{\hat{\g}}(\ell,0)$.
By Proposition \ref{prop:2.5} and equation (\ref{eq:3.13}),
	$W_{\Z}=U_{\Z}(\hat{\mathfrak{g}})U_{\Z}
	=U_{\Z}(\hat{\mathfrak{g}}_{(-)})U_{\Z}$
is the $\Z$-span of elements of the form
	\begin{equation}\label{(eq2.26)}
		a_{i_1}(-n_{i_1})^{l_1}\cdots a_{i_k}(-n_{i_k})^{l_k}b_{j_1}(-n_{j_1})\cdots b_{j_s}(-n_{j_s})u
	\end{equation}
	where
$a_{i}\in (\g_{\Z})_{\bar{0}}$, $b_{j}\in (\g_{\Z})_{\bar{1}}$,
$l_{j}, n_i\in \N$, and
$u\in (U_{\Z})_{\bar{0}}\cup (U_{\Z})_{\bar{1}}$, .
	On the other hand,
recall that
	\begin{equation}\label{(eq33)}
		Y(a(-1)\mathbf{1},x)=\sum_{n\in\Z}a(n)x^{-n-1}\;\;\mbox{for}\;a\in\g.
	\end{equation}
	By Proposition \ref{(prop 2.7)},
	  the vertex subsuperalgebra over $\Z$ $V_{\Z}$ generated by $a(-1)\mathbf{1}$
	  for $a\in(\g_{\Z})_{\bar{0}}\cup (\g_{\Z})_{\bar{1}}$ is
	  the integral span of vectors of the form
	\begin{equation}\label{(eq2.30)}
	a_1(n_1)\cdots a_k(n_k)\mathbf{1}
	\end{equation}
	for $a_i\in(\g_{\Z})_{\bar{0}}\cup (\g_{\Z})_{\bar{1}}$,
	$n_i\in \Z$, $k\in \Z_{+}$.
Notice that $a_i(n)\mathbf{1}=0$ if $n> 0$ and any $a_i(n_i)$
occurring in (\ref{(eq2.30)}) with $n_i>0$ can be moved to the right side
 using the commutation relations (\ref{eq:brac}).
Hence,
 in the case of $U=\C_{\ell}$,
 the set of elements of the form
 \ref{(eq2.30)} is the same as the set of elements of
 the form (\ref{(eq2.26)}).
This proves the first assertion, and the second part follows similarly.
\end{proof}

\begin{coro}\label{(coro3.5)}
	In the setting of Proposition \ref{(prop3.4)},
 $L_{\hat{\g}}(\ell,0)_{\Z}$ is the vertex subsuperalgebra over $\Z$ generated by
 vectors of the form $a(-1)\mathbf{1}$ for $a$ being the homogeneous $\Z$-basis in $\g_{\Z}$.
\end{coro}

From now on, in this subsection,
we assume that
$\g$ is a basic classical Lie superalgebra,
 excluding types $A(1,1)$ and $D(2,1;a)$ with $a\notin \Z$.
Recall the integral form $\tilde{U}_{\Z}(\hat{\g})$ constructed
at the end of Section \ref{sec:2.2},
i.e. $\tilde{U}_{\Z}(\hat{\g})$ is the $\Z$-subalgebra generated by
$
	X_{\beta}(m)^{(r)}$, $X_{\gamma}(m')
$
for $\beta \in \Delta_{\bar{0}}$, $\gamma\in \Delta_{\bar{1}}$,
$m,m'\in \Z$, and $r\in \Z_+$.
There is another way to obtain an integral form
 for $V_{\hat{\g}}(\ell,0)$ and its modules.

\begin{thm}\label{(th3.9)}
	Suppose $\ell \in \Z$.
	Then the integral form $V_{\hat{\g}}(\ell,0)_{\Z}$ is
the vertex subsuperalgebra over $\Z$ of $V_{\hat{\g}}(\ell,0)$ generated by the elements
$X_{\beta}(-1)^{(r)}\mathbf{1}$,
		$X_{\gamma}(-1)\mathbf{1}$,
	where $r\in \Z_{+}$,  and $X_{\beta}$, $X_{\gamma}$ are root vectors in the
	chosen Chevalley basis of $\g$ for $\beta \in \Delta_{\bar{0}}$,
 $\gamma\in \Delta_{\bar{1}}$, respectively.
	Moreover, if $U=U_{\bar{0}}\oplus U_{\bar{1}}$ is a finite-dimensional
 $\g$-module with integral form $U_{\Z}$, and $W$ is either
 $V_{\hat{\g}}(\ell,U)$ or $L_{\hat{\g}}(\ell,U)$,
 then $W_{\Z}$ is the $V_{\hat{\g}}(\ell,0)_{\Z}$-module generated by $U_{\Z}$.
\end{thm}

\begin{proof}
	Since $\tilde{U}_{\Z}(\g)$ is generated as a superalgebra over $\Z$ by the divided powers
	$X_{\beta}(m)^{(n)}$ and $X_{\gamma}(m')$,
	where $\beta\in \Delta_{\bar{0}}$,
	 $\gamma\in \Delta_{\bar{1}}$, $m,m'\in\Z$ and $n\in\Z_{+}$,
	 we can express $W_{\Z}=\tilde{U}_{\Z}(\hat{\g})U_{\Z}$ as the $\Z$-span of the elements of the form
\begin{equation}\label{eq:3.16}
			\begin{split}
			( \prod_{\beta\in \Delta_{\bar{0}} }f_{\beta}) ( \prod_{\gamma\in \Delta_{\bar{1}}} f_{\gamma}) u
			 	=\prod_{\beta\in \Delta_{\bar{0}}} ( X_{\beta}(m_1)^{(n_{\beta,1})}\cdots X_{\beta}(m_p)^{(n_{\beta,p})})
  \prod_{\gamma\in \Delta_{\bar{1}}} ( X_{\gamma}(m_1')\cdots X_{\gamma}(m_q'))  u
			\end{split}
		\end{equation}		
where  $u\in (U_{\Z})_{\bar{0}}\cup (U_{\Z})_{\bar{1}}$,
 $m_1\ge \ldots \ge m_p$, $m_1'\ge \ldots\ge m_q'$,
 $n_{\beta,j}\in \Z_{+}$. (Note that when $U$ is $\C_{\ell}$, $u=\mathbf{1}$.)
By Proposition \ref{(prop 2.7)},
it suffices to show that the $\Z$-span of (\ref{eq:3.16}) is equivalent to
 the $\Z$-span of the coefficients of the form
\begin{equation}
\prod_{\beta\in \Delta_{\bar{0}}}Y(X_{\beta}(-1)^{(n)}\mathbf{1},x) \prod_{\gamma\in \Delta_{\bar{1}}}Y(X_{\gamma}(-1)\mathbf{1},x)u,
\;\;n\in\Z_{+}.
\end{equation}

Now we analyze the vertex operator associated with the generators
$X_{\beta}(-1)^{(n)}\mathbf{1}$ and $X_{\gamma}(-1)\mathbf{1}$
	for $\beta\in \Delta_{\bar{0}}$,
	$\gamma\in \Delta_{\bar{1}}$ and $n\in\Z_{+}$.
	For $\beta\in \Delta_{\bar{0}}$, we have
\begin{equation}
[X_{\beta}(m),X_{\beta}(-1)]=[X_{\beta},X_{\beta}](m-1)+m\delta_{m-1,0}(X_{\beta},X_{\beta})\ell=0.
\end{equation}
Hence,
$
		X_{\beta}(m)X_{\beta}=X_{\beta}(m)X_{\beta}(-1)\mathbf{1}=X_{\beta}(-1)X_{\beta}(m)\mathbf{1}=0$
for $m\in \Z_+$.
	Since the even part of a vertex superalgebra is a vertex algebra,
	by Corollary 3.10.4 of \cite{LL}, we have
	$
		Y(X_{\beta}(-1)^{n}\mathbf{1},x)=Y(X_{\beta},x)^{n}$
for $n\in \Z_+$.
	Subsequently,
$Y(X_{\beta}(-1)^{(n)}\mathbf{1},x)=Y(X_{\beta},x)^{(n)}.
$
	Therefore, by (3.32) of \cite{M1},
	the coefficient of $x^{-k-n} (k\in \Z)$ in $Y(X_{\beta}(-1)^{(n)}\mathbf{1},x)$ is
	\begin{equation}\label{(eq3.34)}
		\sum X_{\beta}(n_1)^{(i_1)}\cdots X_{\beta}(n_m)^{(i_m)},
	\end{equation}
	where $i_1+\cdots+i_m=n$, $n_1i_1+\cdots+n_mi_m=k$,
and the sum is taken over all partitions of $k$ into $n$ parts.
For the vertex operator associated to the generators
$X_{\gamma}(-1)\mathbf{1}$ for $\gamma\in \Delta_{\bar{1}}$,
since
	\begin{equation}\label{(eq3.35)}
		Y(X_{\gamma}(-1)\mathbf{1},x)=\sum_{m'\in \mathbb{Z}}X_{\gamma}(m')x^{-m'-1},
	\end{equation}
	the coefficient of $x^{-m'-1}$ in $Y(X_{\gamma}(-1)\mathbf{1},x)$ is $X_{\gamma}(m')$.
	
	 Consider the case $U=\C\mathbf{1}$, it is evident from
 (\ref{(eq3.34)}) and (\ref{(eq3.35)}) that the vertex subsuperalgebra over $\Z$
 generated by $X_{\beta}(-1)^{(n)}\mathbf{1}$,  $X_{\gamma}(-1)\mathbf{1}$
 for $\beta\in\Delta_{\bar{0}}$, $\gamma\in\Delta_{\bar{1}}$, $n\in\Z_{+}$
  is contained in $V_{\hat{\g}}(\ell,0)_{\Z}$.
	 Conversely, we need to show that
	 $X_{\beta}(m)^{(n)}, X_{\gamma}(m')$ preserve the vertex subsuperalgebra over $\Z$
generated by the vectors $X_{\beta}(-1)^{(n)}\mathbf{1}$, $X_{\gamma}(-1)\mathbf{1}$
 for $\beta\in\Delta_{\bar{0}}$, $\gamma\in\Delta_{\bar{1}}$, $n\in\Z_{+}$.
Note that by (\ref{(eq3.35)}),
	$ X_{\gamma}(m')=(X_{\gamma}(-1){\bf 1})(m^{\prime})$.
Hence, $ X_{\gamma}(m')$ preserve the vertex subsuperalgebra over $\Z$.
 For $\beta\in\Delta_{\bar{0}}$,
we have (see (3.33) of \cite{M1})
	 \begin{equation}
	 	X_{\beta}(m)^{(n)}=(X_{\beta}(-1)^{(n)}\mathbf{1})(mn+n-1)-\sum X_{\beta}(n_1)^{(i_1)}\cdots X_{\beta}(n_m)^{(i_m)},
	 \end{equation}
	 where the sum does not include the partition $(m, \ldots, m)$.
	 Since each $i_j<n$ on the right side, by induction
	  every term on the right side preserves the vertex subsuperalgebra over $\Z$,
	  and so does $X_{\beta}(m)^{(n)}$.
 Since $\mathbf{1}$ is in any vertex subsuperalgebra over $\Z$,
	  (\ref{eq:3.16}) implies that $V_{\hat{\g}}(\ell,0)_{\Z}$ is
 contained in the vertex subsuperalgebra over $\Z$ generated by the
 vectors $X_{\beta}(-1)^{(n)}\mathbf{1}$, $X_{\gamma}(-1)\mathbf{1}$
  for $\beta\in\Delta_{\bar{0}}$, $\gamma\in\Delta_{\bar{1}}$, $n\in\Z_{+}$.
	    In the same way,
	     for any finite-dimensional $\g$ module $U$ with integral
	     form $U_{\Z}$,
	  	 $V_{\hat{\g}}(\ell,U)_{\Z}$ and
	  $L_{\hat{\g}}(\ell,U)_{\Z}$ are the $V_{\hat{\g}}(\ell,0)_{\Z}$-modules generated by $U_{\Z}$.
\end{proof}

By Theorem  \ref{(th3.9)}, and Proposition \ref{(prop 2.7)},
we have the following corollary.
\begin{coro}\label{(coro3.12)}
	The integral form $L_{\hat{\g}}(\ell,0)_{\Z}$ of $L_{\hat{\g}}(\ell,0)$ is
the integral form of $L_{\hat{\g}}(\ell,0)$ as a vertex superalgebra generated
by the vectors $X_{\beta}(-1)^{(n)}\mathbf{1}$, $X_{\gamma}(-1)\mathbf{1},$
 with $\beta\in\Delta_{\bar{0}}$, $\gamma \in \Delta_{\bar{1}}$, $n\in\Z_{+}$.
\end{coro}

\begin{coro}\label{(coro 3.13)}
	If $\g$ is a finite-dimensional simply-laced simple Lie algebra
or a basic classical Lie superalgebra of type $A(m,n)$ ($m,n\neq 1$),
 and $\ell$ is a positive integer,
then the integral form $L_{\hat{\g}}(\ell,0)_{\Z}$ of $L_{\hat{\g}}(\ell,0)$ is
generated by the vectors $X_{\beta}(-1)^{(r)}\mathbf{1}$ and $X_{\gamma}(-1)\mathbf{1}$,
 where $\beta\in \Delta_{\bar{0}}$, $\gamma\in \Delta_{\bar{1}}$, and $0\le r\le \ell$.
\end{coro}
\begin{proof}
	By our assumption of $\g$,
 all even roots of $\g$ have the same length (cf. \cite{H}, \cite{Kac2}, etc.).
	Note that the set of even roots $\Delta_{\bar{0}}$ is the root system of the Lie algebra $\g_{\bar{0}}$.
	Then the result follows from the fact that for any long root
$\beta\in \Delta_{\bar{0}}$, $X_{\beta}(-1)^{\ell+1}.v=0$,
where $v$ is a highest weight vector of a
 $\hat{\g}_{\bar{0}}$-module of level $\ell$ (see Proposition 6.6.4 in \cite{LL}).
 In particular, we have $X_{\beta}(-1)^{\ell+1}.{\bf 1}=0$ for $\beta\in \Delta_{\bar{0}}$.
 With Corollary \ref{(coro3.12)}, we get the result.
\end{proof}

\subsection{The conformal vector in an integral form}

Let $V$ be a vertex operator superalgebra with conformal vector
$\omega$ and central charge $c\in\C$.
 Suppose $V_{\Z}$ is an integral form of $V$.
	 If $V_{\Z}$ contains $k\omega$ where $k\in \C$, then $k^2c\in 2\Z$
	 (see Proposition 5.1 of \cite{M1}).
	 Hence, if $\omega$ is in any integral form of $V$,
the central charge $c$ must be an even integer.
Let $\L=\bigoplus_{n\in \Z}\C L(n)\oplus\C\c$ be the
 Virasoro algebra obtained from the conformal vector $\omega$,
 where $\c$ is central and $L(n)=\omega(n+1)$.
 The Lie bracket is
	 \begin{equation}\label{eq:virel}
	 	[L(m),L(n)]=(m-n)L(m+n)+\frac{m^3-m}{12}\delta_{m+n,0}c,\;\;m,n\in\Z.
	 \end{equation}
	 A vector $v \in V$ is called a {\em lowest weight vector} for $\L$ if
	  $L(0)v = mv$ for some $m \in \Z$ and $L(n)v = 0$ for $n \in \N$.

The following lemma is proved similarly to Lemma 5.3 in \cite{M1},
where the commutator formula is replaced by the super commutator formula (see (2.2.6) in \cite{L}).

\begin{lem}\label{(lem:3.14)}
	For any $m,n\in \Z$ and lowest weight vector $v$,
 $[L(m),v(n)]$ is an integral linear combination of operators $v(k)$ for $k\in \Z$.
\end{lem}

We show that under some assumptions,
there exists an integral form of $V$ containing $k\omega$.

\begin{thm}\label{thomeg}
	Suppose $V_{\Z}$ is an integral form of a vertex operator superalgebra $V$ generated by
	$\Z_{2}\times\Z$-homogeneous lowest weight vectors $\{v^j\}$ for the Virasoro algebra $\mathcal{L}$.
	If $k \in \Z$ is such that $k^2c \in 2\Z$ and $k\omega \in V_{\Q} = \Q \otimes_{\Z} V_{\Z}$,
then $V_{\Z}$ can be extended to an integral form of $V$ containing $k\omega$.
\end{thm}

\begin{proof}	
	Let $\tilde{V_{\Z}}$ be the vertex subsuperalgebra over $\Z$ of $V$ generated by $\{v^j\}$ and $k\omega$.
 It is clear that $V_{\Z}\subseteq \tilde{V_{\Z}}$.
	We will prove that $\tilde{V_{\Z}}$ is also an integral form of $V$.
	By Proposition \ref{(prop 2.7)}, $\tilde{V_{\Z}}$ is the $\Z$-span of products
	\begin{equation}\label{(eq:3.29)}
		u^{1}(m_1)u^{2}(m_2)\ldots u^{r}(m_r)\mathbf{1},
	\end{equation}
	where each $u^i$ is either $v^j$ or $k\omega$.
	By Lemma \ref{(lem:3.14)}, we can rewrite (\ref{(eq:3.29)})
as an integral combination of
	 elements of the form
	\begin{equation}\label{(eq:3.30)}
		(kL(m_1))\cdots (kL(m_j))v^{j_1}(n_1)\cdots v^{j_s}(n_s)\mathbf{1}
	\end{equation}
	for $m_i,n_i\in \Z, v^{j_i}\in \{v^j\}$.
	Note that $\{v^j\}$ and $k\omega$ are both $\Z_{2}$ and $\Z$-homogeneous.
	The $\Z_{2}$-grading of elements of the form (\ref{(eq:3.30)}) is
	$|v^{j_1}|+\ldots+|v^{j_s}|\equiv i \;(\mathrm{mod}\,2)$ for $i \in \{0,1\}$.
	The conformal weight ($L(0)$-grading) of the form (\ref{(eq:3.30)}) is
(recall (\ref{eq:virel}), and (3.1.54) of \cite{LL})
	$-\sum\limits_{t=1}^{j}m_{j}+\sum\limits_{i=1}^{s}(\mbox{wt}(v^{j_i})-n_{i}-1)$.
So $\tilde{V_{\Z}}$ is compatible with the $\Z_2$-grading of $V$ and
the conformal weight grading of $V$.
That $\tilde{V_{\Z}}$ is an integral form of $V$ as a vector space
can be proved the same as Theorem 5.2 in \cite{M1}.
Therefore, $\tilde{V_{\Z}}$ is an integral form of $V$ as
 vertex operator superalgebra
that contains $k\omega$.
\end{proof}

The following is the result for affine vertex operator superalgebras and their modules.
\begin{prop}\label{(prop:3.16)}
	Let $\ell \in \Z$.
	Suppose $V_{\Z}$ is the integral form $V_{\hat{\g}}(\ell,0)_{\Z}$
or $L_{\hat{\g}}(\ell,0)_{\Z}$ of an affine vertex operator superalgebra associated to a
	 basic classical Lie superalgebra $\g$, excluding types $A(1,1)$ and $D(2,1;a)$ with $a\notin \Z$.
	 Then $\omega \in V_{\Q}$ and $V_{\Z}$ is generated by lowest weight vectors for
	 the Virasoro algebra $\mathcal{L}$,
	so $V_{\Z}$ can be extended to an integral form $\tilde{V_{\Z}}$
containing $k\omega$ for any $k\in \Z$ such that $k^2c\in 2\Z$.
\end{prop}

\begin{proof}
	Let $\{u_i\}$ be a Chevalley basis for $\g_{\Z}$ with dual basis $\{u^i\}$
with respect to the form $(\cdot,\cdot)$.
 Take the conformal vector
	\begin{equation*}
		\omega=\frac{1}{2(\ell+h^{\vee})}\sum_{i=1}^{\mathrm{dim}\; \g}(-1)^{\left| u_i \right|}u_i(-1)u^i(-1)\mathbf{1}.
	\end{equation*}
	Since $(\cdot,\cdot)$ is integral on $\g_{\Z}$,
we have $u^i\in \Q \otimes_{\Z}\g_{\Z}$.
With $\ell, h^{\vee}\in \Q$, we get $\omega\in V_{\Q}$.
	Moreover, by Theorem \ref{(th3.9)} and Corollary \ref{(coro3.12)},
 $V_{\Z}$ is generated by the $\Z_2$-homogeneous elements
	$
		X_{\beta}(-1)^{(n)}\mathbf{1}$,
		$X_{\gamma}(-1)\mathbf{1}$
	for $n\in \Z_{+}$, $\beta \in \Delta_{\bar{0}}$, $\gamma\in \Delta_{\bar{1}}$,
	 where $X_{\beta}$, $X_{\gamma}$ are the corresponding root vectors in the Chevalley basis of $\g$.
	By equation (3.3.52) of \cite{Xu}, we have
	\begin{equation}
		[L(m), X_{\alpha}(-1)]=X_{\alpha}(m-1)
	\end{equation}
	for $\alpha \in \Delta$, $m\in \Z$.
	Note that for $\beta\in \Delta_{\bar{0}}$, $m\in\Z$,
$X_{\beta}(m-1)$ commutes with $X_{\beta}(-1)$.
	Then for $m\in \N$, $n\in \Z_+$,
	$\beta\in \Delta_{\bar{0}}$, $\gamma \in \Delta_{\bar{1}}$, we have
	\begin{equation}
\begin{aligned}
& L(m)X_{\beta}(-1)^{(n)}\mathbf{1}=\frac{1}{n}X_{\beta}(-1)^{(n-1)}X_{\beta}(m-1){\bf 1}=0,\\
& L(m)X_{\gamma}(-1)\mathbf{1}=[L(m),X_{\gamma}(-1)]\mathbf{1}=X_{\gamma}(m-1){\bf 1}=0.
\end{aligned}
	\end{equation}
	Since $X_{\beta}(-1)^{(n)}\mathbf{1}$ and $X_{\gamma}(-1){\bf 1}$ are
homogeneous of conformal weight $n$ and $1$, respectively,
	we get that $V_{\Z}$ is generated by $\Z_{2}\times\Z$-homogeneous lowest weight vectors for the Virasoro algebra.
\end{proof}

\subsection{Integral forms in contragredient modules}

 Let $V$ be a vertex operator superalgebra,
  $(W,Y_W)$ a $V$-module
 and $(W',Y_{W'})$
the contragredient module of $W$ (cf. \cite{FHL}, \cite{Xu}, etc.).
It is known that the module $W$ is irreducible if and only if $W'$ is irreducible
(see Proposition 5.3.2 of \cite{FHL}, or \S 3.5 of \cite{Xu}).

Suppose $V$ has an integral form $V_{\Z}$
and $W$ has an integral form $W_{\Z}$.
For $\lambda\in\Z_{2}$,
 define
\begin{equation}
	(W'_{\lambda})_{\Z}=\{w'\in W'_{\lambda}\ | \ \langle w',w\rangle \in \Z \text{ for } w\in (W_{\lambda})_{\Z}\}.
\end{equation}
Let $W'_{\Z}=(W'_{\bar{0}})_{\Z}\oplus (W'_{\bar{1}})_{\Z}$.
Then there is a natural $\Z_2$-grading on $W_{\Z}'$ given by
\begin{equation}\label{(eq:Z2-grad)}
	(W'_{\Z})_{\lambda}:= (W'_{\lambda})_{\Z} \;\;\text{for}\; \lambda \in \Z_2,
\end{equation}
and this grading is compatible with the $\Z_{2}$-grading on $W'$,
i.e. $	(W'_{\Z})_{\lambda}=	W'_{\Z}\cap W'_{\lambda}$.
Furthermore, $W'_{\Z}$ is an integral form
of $W'$ as a super vector space.
We call $W'_{\Z}$ the {\em graded $\Z$-dual} of $W_{\Z}$.

The following two propositions show when
$W'_{\Z}$ is a $V_{\Z}$-module.
Their proofs are similar to the proofs
of Propositions 6.1 and 6.2 of \cite{M1}, respectively.
\begin{prop}\label{(prop4.1)}
	Let $V$ be a vertex operator superalgebra with an integral form $V_{\Z}$. 	
Suppose $V_{\Z}$ is invariant under $\frac{L(1)^n}{n!}$ for $n>0$.
Then $W'_{\Z}$ is invariant under the action of $V_{\Z}$.
\end{prop}

\begin{prop}\label{(prop4.2)}
	Let $V$ be a vertex operator superalgebra with an integral form $V_{\Z}$.
	If $V_{\Z}$ is generated by vectors $v$ such that $L(1)v=0$,
 then $V_{\Z}$ is invariant under $\frac{L(1)^n}{n!}$ for $n\ge 0$.
\end{prop}

Then, by Propositions \ref{(prop:3.16)}, \ref{(prop4.1)} and \ref{(prop4.2)},
we have the following result.
\begin{prop}\label{prop4.3}
Suppose $\ell\in\Z$, $\g$
is a basic classical Lie superalgebra
which is not type $A(1,1)$ or $D(2,1,a)$ with $a\notin\Z$.
	Then graded $\Z$-duals of $V_{\hat{g}}(\ell,0)_{\Z}$
(resp. $(L_{\hat{\g}}(\ell,0))_{\Z}$)-modules are also
 modules of $V_{\hat{\g}}(\ell,0)_{\Z}$ (resp. $L_{\hat{\g}}(\ell,0)_{\Z}$).
\end{prop}

A $V$-module $W$ is equivalent as a $V$-module to its contragredient $W'$
 if and only if there is a non-degenerate bilinear form $( \cdot,\cdot)_{W}$
 on $W$ that is {\em invariant} in the sense that
\begin{eqnarray}
	\begin{aligned}
			&(W_{\lambda}, W_{\mu})_{W}=0\;\; \mbox{if}\;\; \lambda\ne \mu \in \Z_{2}; \\
		&\left(  Y_{W}(v,x)w_1,w_2\right)_{W} =(-1)^{|v||w_1|}\left(  w_1,Y_W(e^{xL(1)}(-x^{-2})^{L(0)}v,x^{-1})w_2\right)_{W} \label{(eq:4.5)}
	\end{aligned}
\end{eqnarray}
for $u\in V_{\bar{0}}\cup V_{\bar{1}}$, $w_1\in W_{\bar{0}}\cup W_{\bar{1}}$, $w_2\in W$
 (see Lemma 3.5.4 in \cite{Xu}, or Remark 5.3.3 of \cite{FHL}, etc.).
 The space of invariant bilinear forms on $V$ is linearly isomorphic to the space
 $\mathrm{Hom}_{\C}(V_{(0)}/L(1)V_{(1)},\C)$
 (Theorem 3.5.10 in \cite{Xu}, or Theorem 3.1 in \cite{L94}).
 If $V$ is an affine vertex operator superalgebra,
 then $V_{(0)}=\C\mathbf{1}$ and $L(1)V_{(1)}=0$.
 So invariant forms on $V$ are unique up to scale.
 And $V$ has a non-degenerate invariant bilinear form only when $V=L_{\hat{\g}}(\ell,0)$
 (see Theorem 3.5.12 of \cite{Xu}, or Theorem 4.9 of \cite{L94}).

 Let $V$ be a vertex operator superalgebra.
 Choose an invariant bilinear form and an integral form $V_{\Z}$ of $V$.
 Then the graded $\Z$-dual $V'_{\Z}$ is an integral form of $V$ as a super
 vector space
 and is invariant under the action of $V_{\Z}$.
 But $V'_{\Z}$ may not be closed under vertex superalgebra products.
 Hence, $V'_{\Z}$ need not be an integral form of $V$ as a vertex superalgebra.
 We have the following result of a relation between $V_{\Z}$ and $V'_{\Z}$
  (cf. the proof of Proposition 6.4 of \cite{M1}).

\begin{prop}\label{prop4.5}
 	Suppose $V$ is equivalent to $V'$ as a $V$-module, $V_{(0)}=\C\mathbf{1}$
 	and $V$ has an integral form $V_{\Z}$ which is invariant under $\frac{L(1)^n}{n!}$ for $n\in \Z_+$.
 Let $\left( \cdot,\cdot \right)_{V}$ be
 	a non-degenerate invariant bilinear form on $V$
 	such that $\left( \mathbf{1},\mathbf{1} \right)\in \Z\setminus \{0\}$.
 	Identify $V_{\Z}'$ with a lattice in $V$ through
 	$\left( \cdot,\cdot \right)_{V}$.
 	 Then $V_{\Z}\subseteq V_{\Z}'$.
 \end{prop}

\end{document}